\documentclass{article}

\usepackage[latin1]{inputenc}
\usepackage{subfigure}
\usepackage{amsmath,amsfonts,amsthm,amssymb}

\def\E{{\mathcal E}}

\def\M{{\mathcal M}}

\def\P{{\mathcal P}}

\def\Prob{{\mathbb{P}}}
\newtheorem{theorem}{Theorem}[section]

\newtheorem{corollary}[theorem]{Corollary}

\author{Stephen Chestnut\footnote{Department of Applied Mathematics, University of Colorado, Boulder, CO 80309-0526, The United States}\and Manuel E. Lladser\footnotemark[1] \footnote{Corresponding author e-mail: manuel.lladser@colorado.edu} \thanks{Both authors have been partially supported by NSF grant \#DMS-0805950.}}
\title{Occupancy distributions in Markov chains via Doeblin's ergodicity coefficient}

\begin{document}
\maketitle
\begin{abstract}
We apply Doeblin's ergodicity coefficient as a computational tool to approximate the occupancy distribution of a set of states in a homogeneous but possibly non-stationary finite Markov chain. Our approximation is based on new properties satisfied by this coefficient, which allow us to approximate a chain of duration $n$ by independent and short-lived realizations of an auxiliary homogeneous Markov chain of duration of order $\ln(n)$. Our approximation may be particularly useful when exact calculations via first-step methods or transfer matrices are impractical, and asymptotic approximations may not be yet reliable. Our findings may find applications to pattern problems in Markovian and non-Markovian sequences that are treatable via embedding techniques. 
 \end{abstract}

\section{Introduction.}
\label{sec:in}

In what follows, $S$ is a given finite set and $T\subset S$ a certain non-empty subset of states. For a fixed integer $n\ge1$, consider a first-order homogeneous Markov chain $X=(X_t)_{0\le t\le n}$ with initial distribution $\mu:S\to[0,1]$ and probability transition matrix $p:S\times S\to[0,1]$. We identify complex-valued functions defined over $S$ and $S\times S$ as row vectors and matrices, respectively. In particular, the distribution of $X_t$ is given by the vector $\mu p^t$.

Our object of interest is the \emph{occupancy distribution of $T$} i.e. the distribution of the random variable:
\begin{equation}
\label{ide:def Sn}
T_n=\sum_{t=1}^n[\![X_t\in T]\!],
\end{equation}
where $[\![\cdot]\!]$ denotes the Iverson's bracket. Random variables of this sort are common in assessing the frequency statistics of patterns in random sequences, which typically model text or genomic sequences. Although various probabilistic and analytic techniques have been used for this purpose, the \emph{Markov chain embedding technique} is among the most versatile ones. This technique seems to have originated in the works of Gerber and Li~\cite{GerLi81}, Biggins and Cannings~\cite{BigCan87}, and Bender and Kochman~\cite{BenKoc93}. It usually consists in embedding a random sequence into the state space of a suitable finite automaton that is informative of the pattern of interest, and it has been completely systematized for \emph{regular patterns} i.e. patterns described by a regular expression, and \emph{Markovian models} of random sequences~\cite{NicSalFla02,Nic03}. In addition, the technique has also shown some promise for assessing regular patterns in \emph{non-Markovian sequences} i.e. sequences with an arbitrary correlation structure~\cite{Lla08}.

All the complexity associated with determining or approximating the distribution of $T_n$ is due to the distributional dependence between the consecutive states visited by the chain $X$. There are various ways---some more ad hoc and others more systematic---to pinpoint this distribution. For small values of $n$, exact calculations are possible via one-step methods~\cite{Dur99} or transfer matrices~\cite{FlaSed09}. Furthermore, transfer matrices lead to Normal approximations for large values of $n$ e.g. as shown in~\cite{NicSalFla02} for frequency statistics of regular patterns under Markovian models. On the other hand, for stationary chains, Poisson~\cite{BarHolJan92} and compound Poisson approximations~\cite{Erh99} have been proposed when $T$ is a \emph{rare set} i.e. the stationary measure of $T$ is small. A specialized instance of these approximations is the P\'olya-Aeppli distribution which occurs as the limiting distribution of frequency statistics of rare words under stationary Markov models~\cite{RoqSch07}.

Our main motivation is to approximate the distribution of $T_n$ in the intermediate regime where $n$ is too large for exact calculations but too small to rely on asymptotic approximations, when $X$ is possibly non-stationary. Our approach relies on a novel probabilistic interpretation and use of the so called \emph{Doeblin's ergodicity coefficient} associated with $p$~\cite{Doe37}, which is defined as:
\begin{equation}
\label{def:Doeblin}
\alpha(p)\stackrel{def}{=}\sum_{j}\min_{i}p(i,j).
\end{equation}

The above motivation is far from artificial! For instance, extensive research is being performed to understand the evolution of complex but short RNA sequences from simpler but functional RNA sequences~\cite{KenLlaYarKni08,KenLlaWuZhaYarDeSKni10}. In contexts like this, the pitfall of the Normal approximation of $T_n$ is the slow rate of convergence of order $n^{-1/2}$. On the other hand, the stationary assumption of the aforementioned Poisson approximations is unrealistic in the context of the Markov chain embedding technique, even if the background model of a genomic sequence is Markovian and stationary. For example, for regular patterns, the initial distribution of the embedded process is concentrated in a few states (of an exponentially large state space) associated with the unique initial state of the (minimal) $k$-th order automaton that recognizes the pattern of interest~\cite{Lla06c}.

%%%%
%%%%
\subsection{\bf Ergodicity coefficients of Markov chains.}
\label{subsec:coefficients}
%%%%
%%%%

This section finishes the Introduction with a brief discussion about ergodicity coefficients and the historical developments surrounding the characterization of weak-ergodicity of inhomogenous Markov chains.

In what follows we denote the set of all probability transition matrices over the state space $S$ as $\P$. The set of all stochastic matrices with identical rows is denoted $\E$; in particular, $\E\subset\P$. We refer to matrices in $\E$ as i.i.d. models because a homogeneous Markov chain with a probability transition matrix in this set is just a sequence of independent and identically distributed (i.i.d.) $S$-valued random variables.

Broadly speaking, an \emph{ergodicity coefficient} is any continuous function $\gamma:\P\to[0,1]$. Such function is said \emph{proper} when $\gamma(p)=1$ if and only if $p\in\E$. Clearly, Doeblin's coefficient as defined in (\ref{def:Doeblin}) is proper. Other ergodicity coefficients found in the literature are:
\begin{eqnarray*}
\gamma_1(p)&\stackrel{def}{=}&\max_j\,\min_i p(i,j);\\
\gamma_2(p)&\stackrel{def}{=}&\min_{i,j}\sum_s\min\{p(i,s),p(j,s)\}=1-\frac{1}{2}\cdot\max_{i,j}\sum_s|p(i,s)-p(j,s)|;\\
\gamma_3(p)&\stackrel{def}{=}&1-\max_s\,\max_{i,j}|p(i,s)-p(j,s)|.
\end{eqnarray*}
Only the last two of these are proper and $\gamma_2$ is called \emph{Markov's ergodicity coefficient}.

Ergodicity coefficients have been proposed for a range of purposes such as to analyze the contractive property of a stochastic matrix~\cite{Mar06} and bound its non-Perron-Froebenius eigenvalues~\cite{Sen93}. However, they have mostly been used to analyze the asymptotic behavior of non-homogenous Markov chains~\cite{Doe37,Haj58,Sen73b}. This entails understanding the asymptotic behavior of products of the form $\prod_{k=m}^n p_k$, with $m\le n$, for a given sequence $(p_k)_{k\ge0}\subset\P$. Such a sequence is said \emph{weakly ergodic} if for all $m\ge0$ and $i,j,s\in S$ the following applies:
\begin{equation}
\label{def:weakerg}
\lim_{n\to\infty}\left|\big(\prod\limits_{k=m}^np_k\big)(i,s)-\big(\prod\limits_{k=m}^np_k\big)(j,s)\right|=0.
\end{equation}
The following condition, known as Markov's theorem~\cite{Sen73a}, is sufficient for weak ergodicity:
\begin{equation}
\label{thm:Markov}
\sum\limits_{k=0}^\infty \gamma_1(p_k)=+\infty.
\end{equation}
This condition is only sufficient in great part because $\gamma_1$ is not proper~\cite{Sen73b}.

In more probabilistic terms, consider a first-order Markov chain $Y=(Y_k)_{k\ge0}$ with state space $S$ such that $\Prob[Y_k=s\mid Y_{k-1},\ldots,Y_0]=p_k(Y_{k-1},s)$, for each $k\ge1$. The sequence $(p_k)_{k\ge0}$ is weakly ergodic if and only if any two independent realizations of $Y$ meet infinitely often on a same state, with probability one. This characterization is due to Doeblin and appeared without proof in the report~\cite{Doe37}. The way in which Doeblin proved this result is matter of speculation and it was lost with his death in World War II (see~\cite{Sen73b} for the historical developments). Furthermore, Doeblin's report remained unnoticed for almost two decades. During this period, the following condition was proved to be both necessary and sufficient for weak ergodicity~\cite{Haj58}:
\begin{equation}
\label{thm:Hajnal}
\begin{array}{l}
\hbox{there\, exists\, a\, strictly\, increasing\, sequence\, of\, positive}\\
\hbox{integers $(n_k)_{k\ge0}$ such that: $\sum\limits_{k=0}^\infty\gamma_2\big(\prod\limits_{i=n_k}^{n_{k+1}-1}p_i\big)=+\infty$.}
\end{array}
\end{equation}
In contrast, Doeblin's characterization of weak ergodicity is the following~\cite{Doe37}:
\begin{equation}
\label{thm:Doeblin}
\begin{array}{l}
\hbox{there\, exists\, a\, strictly\, increasing\, sequence\, of\, positive}\\
\hbox{integers $(n_k)_{k\ge0}$ such that:\,\, $\sum\limits_{k=0}^\infty\alpha\big(\prod\limits_{i=n_k}^{n_{k+1}-1}p_i\big)=+\infty$.}
\end{array}
\end{equation}

Since $\gamma_1(p)\le\alpha(p)\le\gamma_2(p)$, for each $p\in\P$, the sufficient condition in (\ref{thm:Markov}) is a special instance of the conditions in (\ref{thm:Hajnal}) and (\ref{thm:Doeblin}). Though nobody knows how Doeblin proved that conditions (\ref{def:weakerg}) and (\ref{thm:Doeblin}) are equivalent, Seneta ventured in~\cite{Sen73b} a possible proof, that relies on the following two facts, valid for any sequence $(p_k)_{k\ge0}\subset\P$:
\begin{itemize}
\item[(a)] $\Big(1-\gamma_3\big(\prod\limits_{k=0}^np_k\big)\Big)\le\prod\limits_{k=0}^n\big(1-\gamma_2(p_k)\big)$, for all $n\ge0$;
\item[(b)] if $\sum\limits_{k=0}^\infty\alpha\big(p_k\big)=+\infty$ then $\sum\limits_{k=0}^\infty\gamma_1\big(p_k\big)=+\infty$.
\end{itemize}

\hfill\\
\noindent{\bf Paper overlook and organization.} Our paper is mostly about Doeblin's ergodicity coefficient, which we encountered---by accident---while aiming at accurate but low-to-moderate complexity approximations of occupancy distributions in homogenous Markov chains. Here we mostly state and prove new properties about Doeblin's coefficient which we would have never explored otherwise. The more detailed implications of these properties to approximate occupancy distributions will be part of a follow up publication based on the M.S. thesis~\cite{Che10}.

In \S\ref{sec:Doeblin} we demonstrate new properties about Doeblin's coefficient which allow us to provide a new and more elementary proof of Doeblin's characterization of weak-ergodicity (see \S\ref{subsec:Doeblin proof}). In \S\ref{sec:occupancy}, we relate Doeblin's coefficient to a decomposition of the chain into several independent realizations of an auxiliary chain. This leads to a (hopefully) refreshing explanation of the strong-ergodicity of irreducible and aperiodic Markov chains (see \S\ref{subsec:strong ergodicity}). Furthermore, the decomposition allows us to parse (with high probability) the trajectory of a Markov chain of duration $n$ into short-lived realizations of an auxiliary chain of duration of order $\ln(n)$ (see \S\ref{subsec:occupancy}). We exploit this feature to propose new approximations for occupancy distributions based on Doeblin's coefficient, which we compare against Normal and Poisson approximations in a numerical example.

%%%
%%%
\section{A candidate for Doeblin's missing proof.}
\label{sec:Doeblin}
%%%
%%%

Recall that Doeblin's ergodicity coefficient associated with a $p\in\P$ is the quantity:
\[\alpha(p)=\sum_{j}\min_{i}p(i,j).\]
Because $\alpha(\cdot)$ is a proper ergodicity coefficient, $\alpha(p)$ is closed to $1$ when $p$ is in the proximity of some i.i.d. model. However, since the set of i.i.d. models is closed, there should be several i.i.d. models close to $p$. The following result identifies an affine space of i.i.d. models that are in the proximity of $p$.

\begin{theorem}
\label{thm:basic}
For each $p\in\P$, the following applies:
\begin{itemize}
\item[(a)] If $0\le\alpha\le\alpha(p)$ then there is $E\in\E$ and $M\in\P$ such that $p=\alpha\cdot E+(1-\alpha)\cdot M$.
\item[(b)] $\alpha(p)=\sup\left.\big\{\alpha\in[0,1]\,\right|\,(\exists E\in\E)(\exists M\in\P)\!\!: p=\alpha\!\cdot\!E+(1-\alpha)\!\cdot\!M\big\}$.
\item[(c)] Assume that $E\in\E$ and $M\in\P$ are such that $p=\alpha(p)\cdot E+\big(1-\alpha(p)\big)\cdot M$.

If $\alpha(p)<1$ then $\alpha(M)=0$ i.e. $M$ has a zero in each column.

If $\alpha(p)>0$ then $E(i,j)=\frac{1}{\alpha(p)}\cdot\min\limits_s p(s,j)$.

\end{itemize}
\end{theorem}

\begin{proof}
Define $\beta=\alpha(p)$. We first show part (a) in the theorem, for which we may assume without loss of generality that $\beta>0$. In this case, all reduces to prove that there is $E\in\E$ and $M\in\P$ such that
\begin{equation}
\label{ide:part a}
p=\beta\cdot E+(1-\beta)\cdot M.
\end{equation}
Indeed, if $0\le\alpha\le\beta$ then the above implies that $p=\alpha E+(\beta-\alpha) E+(1-\beta) M=\alpha E+(1-\alpha) Q$, for some $Q\in\P$, because 
the matrix $(\beta-\alpha) E+(1-\beta) M$ has nonnegative entries and the sum of the entries in each of its rows is $(1-\alpha)$. To prove the above identity, consider the matrix $E\in\E$ with entries $E(i,j)=\min_s p(s,j)/\beta$. Since $\beta E(i,j)\le p(i,j)$, the matrix $(p-\beta E)$ has nonnegative entries and row sums equal to $(1-\beta)$. In particular, if $\beta=1$ then $p=E$ and the above identity holds with any choice of $M$. Otherwise, it suffices to select $M=(p-\beta E)/(1-\beta)$. This shows (\ref{ide:part a}) and completes the proof of part (a).

To show part (b), notice that if $0\le\alpha\le1$ is such that $p=\alpha E+(1-\alpha) M$, with $E\in\E$ and $M\in\M$, then $\beta=\alpha(p)\ge\alpha\cdot\alpha(E)=\alpha$. Part (b) is now a direct consequence of part (a).

Finally we show part (c). Thus assume that $E\in\E$ and $M\in\P$ such that $p=\beta E+(1-\beta) M$. If $\beta=1$ then $p=E$ and the identity $E(i,j)=\min_s p(s,j)/\beta$ is trivial. On the other hand, if $\beta=0$ then $p$ must have a zero in each column and $M=p$. Without loss of generality we may therefore assume that $0<\beta<1$. We first show that $\alpha(M)=0$. Set $\alpha'=\alpha(M)$. Due to part (a), there exists $E'\in\E$ and $M'\in\P$ such that $p=\beta E+(1-\beta)\alpha' E'+(1-\beta)(1-\alpha')M'$. Hence $\beta=\alpha(p)\ge(\beta+(1-\beta)\alpha')$ and as a result $\alpha'=0$. To complete the proof of the theorem, fix $j$ and notice that $\beta E(i,j)=p(i,j)-(1-\beta)M(i,j)$. In particular, since $M$ has a zero in each column, there is $s$ such that $\beta E(s,j)= p(s,j)$. Finally, since $\beta E(i,j)\le p(i,j)$, we conclude that $\beta E(i,j)=\min_s p(s,j)$. This completes the proof.
\end{proof}

Due to part (a) in the previous theorem, for all $p\in\P$ there is $E\in\E$ and $M\in\P$ such that:
\[(p-E)=(1-\alpha(p))\cdot(M-E).\] 
In particular, the smaller $(1-\alpha(p))$, the closer $p$ is to an i.i.d. model. According to the following result, when one multiplies two or more stochastic matrices, one can only get ``closer'' to the set of i.i.d. models. This is the key ingredient for our proof of Doeblin's characterization of weak ergodicity in the following section.

\begin{corollary}
\label{cor:sublinear}
$\big(1-\alpha(pq)\big)\le\big(1-\alpha(p)\big)\cdot(1-\alpha(q)\big)$, for all $p,q\in\P$.
\end{corollary}
\begin{proof}
Define $\alpha_1=\alpha(p)$ and $\alpha_2=\alpha(q)$. Due to part (a) in Theorem~\ref{thm:basic}, there are matrices $E_1,E_2\in\E$ and $M_1,M_2\in\P$ such that $p=\alpha_1 E_1+(1-\alpha_1) M_1$ and $q=\alpha_2 E_2+(1-\alpha_2) M_2$. In particular, since $p E_2=E_2$, we see that $pq=\alpha_2 E_2+\alpha_1(1-\alpha_2) E_1M_2+(1-\alpha_1)(1-\alpha_2) M_1M_2$. But notice that $E_1M_2\in\E$. Consequently, the rows of the matrix $\alpha_2 E_2+\alpha_1(1-\alpha_2) E_1M_2$ are identical, with common row sum $(\alpha_1+\alpha_2-\alpha_1\alpha_2)$. As a result, there is $E_3\in\E$ such that
\begin{eqnarray*}
pq &=& (\alpha_1+\alpha_2-\alpha_1\alpha_2)\cdot E_3+(1-\alpha_1)(1-\alpha_2)\cdot M_1M_2;\\
&=& (1-(1-\alpha_1)(1-\alpha_2))\cdot E_3+(1-\alpha_1)(1-\alpha_2)\cdot M_1M_2.
\end{eqnarray*}
Finally, due to part (b) in Theorem~\ref{thm:basic}, it follows from the above that
\[\alpha(pq)\ge1-(1-\alpha_1)(1-\alpha_2),\]
which proves the corollary.
\end{proof}

%%%
%%%
\subsection{\bf A first principles proof of Doeblin's characterization of weak ergodicity.}
\label{subsec:Doeblin proof}
%%%
%%%
As we mentioned earlier, Doeblin's proof of his own characterization of weak ergodicity is matter of speculation. Though it is possible to prove that (\ref{def:weakerg}) and (\ref{thm:Doeblin}) are equivalent using Theorem~1 in~\cite{Sen73b} and Corollary~\ref{cor:sublinear}, here we venture an alternative and more elementary proof of this fact. For this fix an integer $m\ge0$ and let $\alpha_n$ denote the Doeblin's ergodicity coefficient of $\prod_{k=m}^np_k$. Due to parts (a) and (c) in Theorem~\ref{thm:basic}, there are matrices $E_n\in\E$ and $M_n\in\P$ such that:
\[\prod_{k=m}^np_k=\alpha_n\cdot E_n+(1-\alpha_n)\cdot M_n,\hbox{ with }\alpha(M_n)=0.\]
In particular, for all $i,j,s\in S$ the following holds:
\begin{equation}
\label{ide:suff Doeblin}
\big(\prod\limits_{k=m}^np_k\big)(i,s)-\big(\prod\limits_{k=m}^np_k\big)(j,s)=(1-\alpha_n)\cdot\big(M_n(i,s)-M_n(j,s)\big).
\end{equation}

Assume first that condition (\ref{thm:Doeblin}) holds. Consider the sets of non-negative integers:
\begin{eqnarray*}
I_n&=&\{k\mid\exists\,j\hbox{ such that $m\le n_j\le k\le n_{j+1}\le n$}\};\\
J_n&=&\{j\mid\exists\,k\in I_n\hbox{ such that }n_j\le k\le n_{j+1}\}.
\end{eqnarray*}
Notice that $J_n$ is an interval of integers. Furthermore, there are stochastic matrices $L$ and $R_n$ such that $\prod_{k=m}^np_k=L\cdot\big(\prod_{k\in I_n}p_k\big)\cdot R_n$. In particular, due to Corollary~\ref{cor:sublinear}, we find that
\[(1-\alpha_n)\le\Big(1-\alpha\big(\prod_{k\in I_n}p_k\big)\Big)\le\prod_{j\in J_n}\Big(1-\alpha\big(\prod_{k=n_j}^{n_{j+1}-1}p_k\big)\Big).\]
Since $(1-x)\le\exp(-x)$, for $0\le x\le1$, the condition in (\ref{thm:Doeblin}) implies that $\lim\limits_{n\to\infty}\alpha_n=1$. Back in (\ref{ide:suff Doeblin}), since each $M_n$ is a stochastic matrix, we conclude that
\[\lim_{n\to\infty}\big(\prod\limits_{k=m}^np_k\big)(i,s)-\big(\prod\limits_{k=m}^np_k\big)(j,s)=0.\]
This shows that condition (\ref{def:weakerg}) is also satisfied i.e. $(p_k)_{k\ge0}$ is weakly ergodic.

Conversely, assume that condition (\ref{def:weakerg}) holds. To show that condition (\ref{thm:Doeblin}) also applies, we first prove that $(\alpha_n)_{n\ge0}$ has a subsequence that converges to one. We show this by contradiction. In particular, due to the identity in (\ref{ide:suff Doeblin}), it applies that
\[\lim_{n\to\infty}\big(M_n(i,s)-M_n(j,s)\big)=0,\]
for all $i,j,s\in S$. Fix $s_1\in S$. Since each $M_n$ has at least one zero in the column associated with $s_1$ then there is $j_1\in S$ and a subsequence $(n_k')_{k\ge0}$ such that $M_{n_k'}(j_1,s_1)=0$ for all $k\ge0$. Therefore, $M_{n_k'}(i,s_1)\to0$ as $k\to\infty$, for all $i\in S$. Now fix $s_2\in S\setminus\{s_1\}$. Since each $M_{n_k'}$ has at least one zero in the column associated with $s_2$ then there is a subsequence $(n_k'')_{k\ge0}$ of $(n_k')_{k\ge0}$ such that $M_{n_k''}(i,s_1)\to0$ and $M_{n_k''}(i,s_2)\to0$ as $k\to\infty$, for all $i\in S$. Since $S$ is finite, a straightforward inductive argument shows that there is a subsequence $(n_k)_{k\ge0}$ such that
\[\lim_{k\to\infty} M_{n_k}(i,s)=0,\]
for all $i,s\in S$. However, the above is not possible because each $M_{n_k}$ is a stochastic matrix. As a result, $(\alpha_n)_{n\ge0}$ must have a subsequence that converges to one.

The previous argument shows that if $(p_k)_{k\ge0}$ is weakly ergodic then, for all $m\ge0$, there is $n\ge m$ such that e.g. $\alpha(\prod_{k=m}^n p_k)\ge1/2$. From this, condition (\ref{thm:Doeblin}) is immediate and we have proved that conditions (\ref{def:weakerg}) and (\ref{thm:Doeblin}) are equivalent.

%%%
%%%
\section{Occupancy distributions of homogeneous chains.}
\label{sec:occupancy}
%%%
%%%

In this section we retake our original motivation of approximating occupancy distributions in finite homogeneous Markov chains. For this notice that all the complexity associated with computing or approximating occupancy distributions is due to the distributional dependence between consecutive transitions of $X=(X_i)_{i\ge0}$. Our next result yields a stochastic equivalent of $X$ based on Doeblin's ergodicity coefficient that breaks (at random times) this dependence. To state the result, assume that:
\begin{equation}
\label{ide:palpha}
p=\alpha\cdot E+(1-\alpha)\cdot M,
\end{equation}
for certain $0\le\alpha\le1$, $E\in\E$ and $M\in\P$, and denote as $\mathbf{e}$ any of the rows of $E$.

\begin{theorem}
\label{thm:equivalent}
Assume that condition (\ref{ide:palpha}) is satisfied. Imagine a coin that shows $E$ with probability $\alpha$ and $M$ with probability $(1-\alpha)$ when tossed. The stochastic sequence $Y=(Y_i)_{i\ge0}$ defined as follows:
\begin{itemize}
\item[(i)] $Y_0$ has distribution $\mu$, and
\item[(ii)] for each $i\ge0$, the distribution of $Y_{i+1}$ conditioned on $(Y_0,\ldots,Y_i)$ is given by the following procedure: toss the coin, and if the $E$-side comes up then draw $Y_{i+1}$ using the distribution $\mathbf{e}(\cdot)$, else draw $Y_{i+1}$ using the distribution $M(Y_i,\cdot)$,
\end{itemize}
has the same distribution as $X$.
\end{theorem}

\begin{proof}
Due to the definition of the $Y$ process, for each $i\ge0$ and $s_0,\ldots,s_{i+1}\in S$, the following applies:
\begin{eqnarray*}
\Prob(Y_{i+1}=s_{i+1}\mid Y_0=s_0,\ldots,Y_i=s_i)&=&\alpha\cdot e(s_{i+1})+(1-\alpha)\cdot M(s_i,s_{i+1}),\\
&=&\alpha\cdot E(s_i,s_{i+1})+(1-\alpha)\cdot M(s_i,s_{i+1})=p(s_i,s_{i+1}).
\end{eqnarray*}
In particular, $Y$ is a first-order homogeneous Markov chain with initial distribution $\mu$ and probability transition matrix $p$. Hence $X$ and $Y$ have the same distribution.\end{proof}

Next, we show how to exploit the random times at which the imaginary coin of the theorem breaks the dependence between consecutive transitions of the $Y$-chain. The first application gives a non-standard argument for the strong ergodicity of irreducible and aperiodic Markov chains. In the second application, we refine this argument to obtain low-to-moderate complexity approximations of occupancy distributions.

For what follows, recall that the \emph{total variation distance} between two probability distributions $u(\cdot)$ and $v(\cdot)$ supported over $\mathbb{N}=\{0,1,\ldots\}$ is defined as:
\[\|u-v\|\stackrel{def}{=}\sup_{A\subset\mathbb{N}}|u(A)-v(A)|=\frac{1}{2}\sum_{i\in\mathbb{N}}|u(i)-v(i)|,\]
where $u(i)$ and $v(i)$ denote $u(\{i\})$ and $v(\{i\})$, respectively. Accordingly, the total variation distance between two $\mathbb{N}$-valued random variables $U$ and $V$ is defined as the total variation distance of their distributions and it is denoted $\|U-V\|$.

%%%
%%%
\subsection{Connections with strong ergodicity.}
\label{subsec:strong ergodicity}
%%%
%%%

It is well-known that if $p$ is irreducible and aperiodic then there is a unique stationary distribution i.e. a unique probability distribution $\pi$ such that $\pi p=\pi$. In this case, there are constants $c_0,c_1>0$ which depend on $p$ but not on $\mu$, such that:
\begin{equation}
\label{ide:stationarity}
\|X_n-\pi\|\le c_0e^{-c_1\cdot n},\hbox{ for all $n\ge0$}.
\end{equation}
In particular, the distribution of $X_n$ is asymptotically independent of $n$. Using Theorem~\ref{thm:equivalent}, one may alternatively explain this phenomena as follows. If $\alpha(p)>0$ then there is a distribution $\mathbf{e}$ such that, regardless of the state where the chain is located, the next state will be picked up from this distribution with probability $\alpha(p)$. Each time this distribution is used, any information about the states previously visited by the chain is lost. This distribution acts therefore as a ``memory-breaker''. When $n$ is large, and even if $\alpha(p)$ is small, it is unlikely that no memory-breaker occurred between $Y_0$ and $Y_n$. Since all transitions after the last memory-breaker where controlled by $M$, the distribution of $Y_n$ should be well-approximated by a mixture of distributions of the form $\mathbf{e}M^t$. This intuition is made precise on the following result. Due to part (b) in Theorem~\ref{thm:basic}, notice that the optimal choice for $\alpha$ is $\alpha(p)$.

\begin{corollary}
\label{cor:stationarity}
\cite{Che10} Assume that condition (\ref{ide:palpha}) is satisfied. Let $m\ge0$ and consider $S$-valued random variables $Z_0,\ldots,Z_m$ such that $Z_t$ has distribution $\mathbf{e}M^t$. If $I$ is a random index independent of $(Z_1,\ldots,Z_m)$ such that $\Prob[I=t]=\alpha(1-\alpha)^t/(1-(1-\alpha)^{m+1})$, for $0\le t\le m$; in particular, $\Prob[I\in\{0,\ldots,m\}]=1$, then
\[\|X_n-Z_I\|\le(1-\alpha)^{m+1},\hbox{ for all $n>m$}.\]
In particular, if $\alpha>0$ and $p$ is irreducible and aperiodic then $\pi=\alpha\,\mathbf{e}\,(\mathbb{I}-\big(1-\alpha)M\big)^{-1}$ and
\[\|X_n-\pi\|\le2(1-\alpha)^n,\hbox{ for all $n\ge1$}.\]
\end{corollary}

To fix ideas, consider the probability transition matrix:
\begin{equation}
\label{ide:exa:p}
p=\left[\begin{array}{ccc}
\frac{3}{10}&0&\frac{7}{10}\vspace{3pt}\\
0&\frac{9}{10}&\frac{1}{10}\vspace{3pt}\\
\frac{4}{5}&\frac{1}{5}&0\end{array}\right].
\end{equation}
Since $\alpha(p)=0$, the inequalities of the corollary are trivial. However, observe that:
\[p^2=\left[\begin{array}{ccc}
{\frac {13}{20}}&{\frac {7}{50}}&{\frac {21}{100}}\vspace{3pt}\\
{\frac {2}{25}}&{\frac {83}{100}}&{\frac {9}{100}}\vspace{3pt}\\
{\frac {6}{25}}&{\frac {9}{50}}&{\frac {29}{50}}\end {array} \right]=\frac{31}{100}\cdot E_2+\frac{69}{100}\cdot M_2,\]
with
\[E_2:=\left[ \begin {array}{ccc}
{\frac {8}{31}}&{\frac {14}{31}}&{\frac {9}{31}}\vspace{3pt}\\
{\frac {8}{31}}&{\frac {14}{31}}&{\frac {9}{31}}\vspace{3pt}\\
{\frac {8}{31}}&{\frac {14}{31}}&{\frac {9}{31}}\end {array} \right]\hbox{ and }\,
M_2:=\left[ \begin {array}{ccc}
{\frac {19}{23}}&0&{\frac {4}{23}}\vspace{3pt}\\
0&1&0\vspace{3pt}\\
{\frac {16}{69}}&{\frac {4}{69}}&{\frac {49}{69}}\end {array} \right].\]

Notice that $\alpha_2:=\alpha(p^2)=31/100$; in particular, the above decomposition of $p^2$ is a direct consequence of part (c) in Theorem~\ref{thm:basic}. Define $\mathbf{e}_2$ as the first row of $E_2$. Imagine you would like to approximate the distribution of some $X_n$, with as few matrix multiplications as possible, and within a $5\%$ accuracy in total variation distance. Define $\epsilon:=0.05$. Due to the Corollary~\ref{cor:stationarity}, this is possible for any even number $n\ge18$, by considering a mixture of the distributions $\mathbf{e}_2\,M_2^t$, with $t=0,\ldots,8$. On the other hand, because Markovian kernels are contractive in total variation distance~\cite{Mar06}, this is also possible for any odd number $n\ge18$ by considering a mixture of the distributions $\mathbf{e}_2\,M_2^t\,p$, with $t=0,\ldots,8$. Either mixture can be computed in at most $10$ matrix multiplications, however, as seen in Table~\ref{table:stationary}, this number can be optimized by considering larger powers of $p$. Indeed, it is possible to approximate within $\epsilon$-units the distribution of each $X_n$, with $n\ge16$, using a mixture of three distributions associated with Doeblin's ergodicity coefficient of $p^4$. This mixture is given by:
\begin{equation}
\label{ide:ZIp4}
\sum_{t=0}^3\frac{(1-\alpha_4)^t-(1-\alpha_4)^{t+1}}{1-(1-\alpha_4)^4}\cdot \mathbf{e}_4\cdot M_4^t\cdot p^{n(\!\!\!\!\!\!\mod 4)},
\end{equation}
which can be computed using $7$ matrix multiplications. In retrospect, this is far from obvious. For instance, using a computer algebra, one finds that:
\[p^7=\left[\begin {array}{ccc} 
0.3444507000& 0.3440640000& 0.3114853000\\
0.1966080000& 0.6114381000& 0.1919539000\\
0.3559832000& 0.3839078000& 0.2601090000
\end {array} \right].\]
Since $\max_{i,j}\|p^7(i,\cdot)-p^7(j,\cdot)\|\ge0.21$, the chain is still far from its stationary distribution even after 7 transitions. Indeed, $\max_i\|p^7(i,\cdot)-\pi(\cdot)\|\ge 0.13$, exceeding the total variation distance between any $X_n$, with $n\ge16$, and the distribution in (\ref{ide:ZIp4}).

\begin{table}[h!]
\begin{center}
\begin{small}
\begin{tabular}{ccccc}
$k$ & $\alpha_k$ & $m_k$ & $n_k$ & $c_k$\\
\hline
$1$ & 0 & $\infty$ & $\infty$ & $\infty$ \\
$2$ & 0.31 & 8 & 18 & 10\\
$3$ & 0.403 & 5 & 18 & 8\\
$4$ & 0.5287 & 3 & 16 & 7\\
$5$ & 0.63234 & 3 & 20 & 8\\
$6$ & 0.758471 & 2 & 18 & 8\\
$7$ & 0.857157 & 2 & 21 & 9\\
%$8$ & 0.92319959 & 1 & 16 & 9\\
%$9$ & 0.93796191 & 1 & 18 & 10\\
%$10$ & 0.9427713551 & 1 & 20 & 11
\end{tabular}
\end{small}
\end{center}
\caption{Parameters associated with powers of the probability transition matrix in (\ref{ide:exa:p}). Here $\alpha_k=\alpha(p^k)$ and $p^k=\alpha_kE_k+(1-\alpha_k)M_k$, with $E_k\in\E$ and $M_k\in\P$. In addition, $m_k:=\lceil\ln(\epsilon)/\ln(1-\alpha_k)-1\rceil$, with $\epsilon:=0.05$, and $n_k:=k(m_k+1)$. Due to Corollary~\ref{cor:stationarity}, if $\mathbf{e}_k$ denotes any of the rows of $E_k$ then, for each $n\ge n_k$, there exists a mixture of the distributions $\mathbf{e}_k\,M_k^t\,p^{n(\!\!\!\!\mod k)}$; $t=0,\ldots,m_k$, which approximates the distribution of $X_n$ within $\epsilon$-units in total variation distance. This mixture can be computed with at most $c_k:=(k+m_k)$ matrix multiplications.}
\label{table:stationary}
\end{table}

%%%
%%%
\subsection{Approximation of occupancy distributions.}
\label{subsec:occupancy}
%%%
%%%
Assume that condition (\ref{ide:palpha}) is satisfied. Following the notation of Theorem~\ref{thm:equivalent}, the occupancy distribution of a set $T\subset S$ is the distribution of the random variable:
\[T_n=\sum_{t=1}^n[\![Y_t\in T]\!].\]
The moment generating function (m.g.f.) of $T_n$ is given by $\mu\cdot\{p(z)\}^n\cdot\mathbf{1}$, where $p(z)$ is the matrix with polynomial entries given by $p(z)(i,j)=p(i,j)\cdot z^{[\![j\in T]\!]}$, and $\mathbf{1}$ is a column-vector of ones. $p(z)$ is called a transfer matrix~\cite{FlaSed09}, and the computation of the exact distribution of $T_n$ is expensive unless $n$ is relatively small. In what follows, we extend the argument of the previous section to approximate this distribution.

Notice that the random variables $[\![Y_i\in T]\!]$ and $[\![Y_j\in T]\!]$, with $i<j$, are independent when at least one of the random variables $Y_{i+1},\ldots,Y_j$ is drawn from the the memory-breaker distribution $\mathbf{e}$. In particular, the times at which the $E$-side of the coin appears cut the trajectory $Y_0,\ldots,Y_n$ into independent ``pieces''. The number of such pieces is random, and consecutive transitions in each piece are governed by the matrix $M$. Furthermore, the initial distribution of each piece is $\mathbf{e}$, except for the first piece which has initial distribution $\mu$. The expected number of memory-breakers between the first and last transition is $\alpha n$; and the average separation between consecutive memory-breakers is $1/\alpha$, regardless of $n$. As a result, a mixture of $\mathbf{e}(z)\cdot\{M(z)\}^m\cdot\mathbf{1}$, with $m$ an integer in a neighborhood of $1/\alpha$, should lead to a decent approximation of the m.g.f. of the occupancy distribution of $T$ in each piece other than the first. For the first piece, the m.g.f.'s to consider are of the form $\mu\cdot\{M(z)\}^m\cdot\mathbf{1}$. Since the behavior of the Markov chain is independent from one piece to another, an approximation for the m.g.f. of $T_n$ should follow. More importantly for computations, a power of order $o(n)$ of the transfer matrix $M(z)$ should suffice for a decent approximation of the distribution of $T_n$. The weakest point of this heuristic is the probable occurrence of longer than expected pieces at already intermediate values of $n$. This motivates us to look at the random variable $L_n$ defined as the length of the longest piece. (In probabilistic terms, $L_n$ is the length of the largest run of $M$'s in $n$-tosses of the coin from Theorem~\ref{thm:equivalent}.) The asymptotic distribution of this random variable is well understood, both via combinatorial and probabilistic methods~\cite{Fel68,FlaSed09,Arr90}. Since the distribution of $L_n$ concentrates around $-\ln(\alpha n)/\ln(1-\alpha)$ as $n$ increases, selecting $m=-c\ln(\alpha n)/\ln(1-\alpha)$, for $c>1$, gives $\Prob[L_n\le m]=1+\mathcal{O}(n^{1-c})$. An explicit upper-bound for the error in total variation distance follows now from the next result. We notice that the m.g.f. of the random variable $W_I$ in the corollary can be computed explicitly via a symbolic specification~\cite{FlaSed09}.

\begin{corollary}
\label{cor:equivalent2}
\cite{Che10} Assume that condition (\ref{ide:palpha}) is satisfied.
Fix $m\geq0$ and define $\ell_n=\Prob[L_n\le m]$. Let
$I=(I_0,\ldots,I_K)$ be a random composition of $n$ such that
$\Prob[I=(i_0,i_1,\ldots,i_k)]=\alpha^k(1-\alpha)^{n-k}/\ell_n$, for
all $k\ge0$, $0\le i_0\le m$, $1\le i_l\le(m+1)$, for $l\ge1$, and such that $\sum_{l=0}^ki_l=n$. In
addition, consider independent random variables $U(l)$, $V(i,l)$ which are independent of $I$ and such that $U(l)$
has m.g.f. $\mu\cdot\{M(z)\}^l\cdot\mathbf{1}$ and $V(i,l)$ has m.g.f.
$\mathbf{e}(z)\cdot\{M(z)\}^{l-1}\cdot\mathbf{1}$. If one defines
$W_I:=U(I_0)+\sum_{l=1}^K V(l,I_l)$ then
\begin{equation}\label{eq:rvApproximation}
\|T_n - W_I\|\leq \frac{1-\ell_n}{\ell_n}.
\end{equation}
\end{corollary}

As a numerical example, we select a stationary and homogenous Markov chain from~\cite{Erh99} with state space $S=\{1,\ldots,8\}$ and probability transition matrix
\begin{equation}\label{eq:ErhardssonEx1}
p(i,j) = \left\{ \begin{array}{lcl}
\frac{1-\beta q(i,T)}{1-q(i,T)}q(i,j) &, & i\in T^c,j\in T^c;\\
\beta q(i,j) &, & i\in T^c, j\in T;\\
q(i,j) &,& i\in T, j\in S;
\end{array}\right.
\end{equation}
where
\begin{equation*}
q=\left(\begin{array}{cccccccc}
0.334 & 0.215 & 0.173 & 0.119 & 0.065 & 0.086 & 0.003 & 0.005 \\
0.289 & 0.133 & 0.211 & 0.133 & 0.067 & 0.156 & 0.007 & 0.004 \\
0.356 & 0.184 & 0.075 & 0.043 & 0.151 & 0.183 & 0.002 & 0.006 \\
0.41    & 0.162 & 0.108 & 0.075 & 0.14  & 0.097 & 0.005 & 0.003 \\
0.316 & 0.239 & 0.044 & 0.218 & 0.076 & 0.098 & 0.004 & 0.005 \\
0.44    & 0.176 & 0.044 & 0.242 & 0.088 & 0             & 0.005 & 0.005 \\
0.18    & 0.06  & 0.19  & 0.09  & 0.13  & 0.1   & 0.13  & 0.12 \\
0.2     & 0.16  & 0.07  & 0.1   & 0.14  & 0.1   & 0.09  & 0.14
\end{array}\right).
\end{equation*}
The goal is to approximate the occupancy distribution of the set $T=\{8\}$ for various values of $n$ and $\beta$. The parameter $\beta$ controls transitions to $T$, which become rare for $\beta$ small. Table \ref{tab:dTV} gives exact total variation distance errors for Normal~\cite{FlaSed09} and compound Poisson approximations~\cite{Erh99} as well as our approximation in (\ref{eq:rvApproximation}). As shown in the table, approximation~\eqref{eq:rvApproximation} gives one order of magnitude or more improvement over the compound Poisson approximation.  Furthermore, it is clear that $n=1000$ may be not large enough for an accurate Normal approximation to the occupancy distribution of $T$.
\begin{table}[h!]
\begin{center}
\begin{small}
\begin{tabular}{ccccc}
$n$ & $\beta$ & Normal Approximation & Compound Poisson Approximation &
Approximation in \eqref{eq:rvApproximation}\\
\hline
10      &       1                       &1.7E-2         &3.2E-3         &3.1E-4\\
10      &       0.5             &1.7E-2         &1.2E-3         &1.3E-4\\
10      &       0.2             &1.3E-2         &4.9E-4         &5.6E-5\\
10      &       0.1             &5.3E-3         &1.7E-4         &2.1E-5\\
10      &       0.01    &5.3E-4         &1.5E-5         &2.1E-6\\
10      &       0                       &5.3E-5         &1.5E-6         &2.1E-7\\
100     &       1                       &0.23                   &9.7E-3         &2.3E-4\\
100     &       0.5             &0.22                   &3.5E-3         &1.6E-4\\
100     &       0.25    &0.14                   &1.3E-3         &7.5E-5\\
100     &       0.1             &2.0E-2         &3.1E-4         &3.1E-5\\
100     &       0.01    &5.2E-3         &1.6E-5         &3.3E-6\\
100     &       0                       &5.3E-4         &1.5E-6         &3.3E-7\\
1000&   1                       &6.9E-2         &9.4E-3         &2.1E-5\\
1000&   0.5             &9.0E-2         &4.9E-3         &1.4E-5\\
1000&   0.25    &0.14                   &2.7E-3         &8.2E-6\\
1000&   0.1             &0.23                   &9.6E-4         &1.1E-5\\
1000&   0.01    &2.0E-2         &2.7E-5         &1.8E-6\\
1000&   0                       &5.2E-3         &1.7E-6         &2.0E-7\\
\end{tabular}
\end{small}
\caption{Total variation distance for approximations to the occupancy distribution of the set $T=\{8\}$ for the stationary chain described by \eqref{eq:ErhardssonEx1}.  The compound Poisson approximation, given by~\cite{Erh99}, is a P\'olya-Aeppli distribution.}
\label{tab:dTV}
\end{center}
\end{table}

\section*{Acknowledgements}
The authors are indebted to Emeritus Professor Eugene Seneta, who graciously provided a hardcopy of Doeblin's 1937 report.
%\verb!\acknowledgements!

%\nocite{*}
\bibliographystyle{plain}
\bibliography{biblio}

\begin{thebibliography}{10}

\bibitem{Arr90}
R.~Arratia, L.~Goldstein, and L.~Gordon.
\newblock Poisson approximation and the {C}hen-{S}tein method.
\newblock {\em Statistical Science}, 5(4):403--424, 1990.

\bibitem{BarHolJan92}
A.~D. Barbour, L.~Holst, and S.~Janson.
\newblock {\em Poisson Approximation}.
\newblock Oxford University Press, first edition, 1992.

\bibitem{BenKoc93}
E.~A. Bender and F.~Kochman.
\newblock The distribution of subword counts is usually normal.
\newblock {\em Eur. J. Comb.}, 14(4):265--275, 1993.

\bibitem{BigCan87}
J.~D. Biggins and C.~Cannings.
\newblock {M}arkov renewal processes, counters and repeated sequences in
  {M}arkov chains.
\newblock {\em Adv. Appl. Prob.}, 19:521--545, 1987.

\bibitem{Che10}
S.~Chestnut.
\newblock Approximating {M}arkov chain occupancy distributions.
\newblock Master's thesis, University of Colorado, The United States, April
  2010.

\bibitem{Doe37}
W.~Doeblin.
\newblock Le cas discontinu des probabilit{\'e}s en cha{\^i}ne.
\newblock {\em Publ. Fac. Sci. Univ. Masaryk (Brno)}, 236:1--13, 1937.

\bibitem{Dur99}
R.~Durrett.
\newblock {\em Essentials of stochastic processes}.
\newblock Springer, 1999.

\bibitem{Erh99}
T.~Erhardsson.
\newblock Compound {P}oisson approximation for {M}arkov chains using {S}tein's
  method.
\newblock {\em Ann. Prob.}, 27:565--596, 1999.

\bibitem{Fel68}
W.~Feller.
\newblock {\em An Introduction to Probability Theory and Its Applications}.
\newblock John Wiley \& Sons, third edition, 1968.

\bibitem{FlaSed09}
P.~Flajolet and R.~Sedgewick.
\newblock {\em Analytic Combinatorics}.
\newblock Cambridge University Press, first edition, 2009.

\bibitem{GerLi81}
H.~U. Gerber and S.-Y.~R. Li.
\newblock The occurrence of sequence patterns in repeated experiments and
  hitting times in a {M}arkov chain.
\newblock {\em Stochastic Processes and their Applications}, 11(1):101--108,
  1981.

\bibitem{Haj58}
J.~Hajnal.
\newblock Weak ergodicity in nonhomogeneous {M}arkov chains.
\newblock {\em Proc. Cambridge. Philos. Soc.}, 54:233--246, 1958.

\bibitem{KenLlaWuZhaYarDeSKni10}
R.~Kennedy, M.~E. Lladser, Z.~Wu, C.~Zhang, M.~Yarus, H.~De Sterck, and
  R.~Knight.
\newblock Natural and artificial {RNA}s occupy the same restricted region of
  sequence space.
\newblock {\em RNA}, 16(2):280--289, 2010.

\bibitem{KenLlaYarKni08}
R.~Kenney, M.~E. {Ll}adser, M.~Yarus, and R.~Knight.
\newblock Information, probability, and the abundance of the simplest {RNA}
  active sites.
\newblock {\em Front Biosci.}, 13:6060--71, 2008.

\bibitem{Lla06c}
M.~E. Lladser.
\newblock Minimal {M}arkov chain embeddings of pattern problems.
\newblock In {\em Proceedings of the 2007 Information Theory and Applications
  Workshop}, University of California, San Diego, 2006.

\bibitem{Lla08}
M.~E. Lladser.
\newblock Markovian embeddings of general random strings.
\newblock In {\em 2008 Proceedings of the Fifth Workshop on Analytic
  Algorithmics and Combinatorics}, pages 183--190, San Francisco, California,
  2008. SIAM.

\bibitem{Mar06}
A.~A. Markov.
\newblock Extension of the law of large numbers to dependent quantities (in
  {R}ussian).
\newblock {\em Izvestiia Fiz.-Matem. Obsch. Kazan Univ.}, 15:135--156, 1906.
\newblock (2nd Ser.).

\bibitem{Nic03}
P.~Nicod{\`e}me.
\newblock Regexpcount, a symbolic package for counting problems on regular
  expressions and words.
\newblock {\em Fundamenta Informaticae}, 56(1-2):71--88, 2003.

\bibitem{NicSalFla02}
P.~Nicod{\`e}me, B.~Salvy, and P.~Flajolet.
\newblock Motif statistics.
\newblock {\em Theoretical Computer Science}, 287(2):593--617, 2002.

\bibitem{RoqSch07}
E.~Roquain and S.~Schbath.
\newblock Improved compound poisson approximation for the number of occurrences
  of any rare word family in a stationary markov chain.
\newblock {\em Adv. in Appl. Probab.}, 39(1):128--140, 2007.

\bibitem{Sen73a}
E.~Seneta.
\newblock {\em Non-negative matrices}.
\newblock John Wiley \& Sons, first edition, 1973.

\bibitem{Sen73b}
E.~Seneta.
\newblock On the historical development of the theory of finite inhomogeneous
  {M}arkov chains.
\newblock {\em Mathematical Proceedings of the Cambridge Philosophical
  Society}, 74:507--513, 1973.

\bibitem{Sen93}
E.~Seneta.
\newblock Applications of ergodicity coefficients to homogeneous {M}arkov
  chains.
\newblock In {\em Doeblin and Modern Probability}, volume 149 of {\em
  Contemporary Mathematics}, pages 189--199. AMS, 1993.

\end{thebibliography}
\label{sec:biblio}

\end{document}